\documentclass{amsart}  

\usepackage{amsrefs, multirow, bigdelim}
\usepackage{xcolor}
\usepackage{amsfonts}
\usepackage{amsthm}
\usepackage{amsmath}
\usepackage{amsfonts}
\usepackage{latexsym}
\usepackage{amssymb}
\usepackage[latin1]{inputenc}
\usepackage{verbatim}
\usepackage{enumerate}
\usepackage{comment}
\usepackage{graphicx}
\usepackage{url}
\usepackage{tensor}
\usepackage[margin=1in]{geometry}
\usepackage{tikz}

\newtheorem{theorem}{Theorem}[section]
\newtheorem{lemma}[theorem]{Lemma}
\newtheorem{proposition}[theorem]{Proposition}
\newtheorem{corollary}[theorem]{Corollary}
\newtheorem{example}[theorem]{Example}
\newtheorem{definition}[theorem]{Definition}

\newtheorem{remark}[theorem]{Remark}

\theoremstyle{definition}

\newcommand{\bzero}{{\bf 0}}
\newcommand{\bone}{{\bf 1}}
\newcommand{\bh}{{\bf h}}

\newcommand{\iu}{\mathrm{i}\mkern1mu} 

\newcommand{\ZZ}{\mathbb{Z}}
\newcommand{\RR}{\mathbb{R}}
\newcommand{\QQ}{\mathbb{Q}}

\newcommand{\ii}{\mathtt{i}}

\newcommand{\cS}{\mathcal{S}}
\newcommand{\cR}{\mathcal{R}}
\newcommand{\cI}{\mathcal{I}}
\newcommand{\cL}{\mathcal{L}}

\newcommand{\e}{\mathrm{e}}
\newcommand{\I}{\texttt{I}}
\newcommand{\J}{\texttt{J}}

\begin{document}

\title{Complex Hadamard Diagonalisable Graphs}

\author[A.~Chan]{Ada Chan\textsuperscript{1}}
\thanks{\textsuperscript{1}Department of Mathematics and Statistics, York University, Toronto, ON, M3J 1P3, Canada. {\tt ssachan@yorku.ca}}
\author[S.~Fallat]{Shaun Fallat\textsuperscript{2}}
\thanks{\textsuperscript{2}Department of Mathematics and Statistics, University of Regina, Regina, SK, S4S 0A2, Canada. {\tt shaun.fallat@uregina.ca}}
\author[S.~Kirkland]{Steve Kirkland\textsuperscript{3}}
\thanks{\textsuperscript{3}Department of Mathematics, University of Manitoba, Winnipeg, MB, R3T 2N2, Canada. 
{\tt stephen.kirkland@umanitoba.ca}}
\author[J.C.-H.~Lin]{Jephian C.-H.\ Lin\textsuperscript{4}}
\thanks{\textsuperscript{4}Department of Mathematics and Statistics, University of Victoria, Victoria, BC V8W 2Y2, Canada and Department of Applied Mathematics, National Sun Yat-sen University, Kaohsiung, Taiwan, 80424. {\tt jephianlin@gmail.com}}
\author[S.~Nasserasr]{Shahla Nasserasr\textsuperscript{5}}
\author[S.~Plosker]{Sarah Plosker\textsuperscript{5}}
\thanks{\textsuperscript{5}Department of Mathematics and Computer Science, Brandon University,
Brandon, MB R7A 6A9, Canada. {\tt nasserasrs@Brandonu.ca; ploskers@brandonu.ca}}



\begin{abstract}  
In light of recent interest in Hadamard diagonalisable graphs (graphs whose Laplacian matrix is diagonalisable by a Hadamard matrix), we generalise this notion from real to complex Hadamard matrices. We give some basic properties  and methods of constructing such graphs. We   show that a  large class of complex Hadamard diagonalisable graphs have vertex sets   forming an equitable partition, and that the Laplacian eigenvalues must be even integers. We provide a number of examples and constructions  of complex Hadamard diagonalisable graphs, including two special classes of graphs: the Cayley graphs over $\mathbb{Z}_r^d$, and  the non--complete extended $p$--sum (NEPS).  
We discuss necessary and sufficient conditions for  $(\alpha, \beta)$--Laplacian fractional revival and perfect state transfer on continuous--time quantum walks described by   complex Hadamard diagonalisable graphs and provide examples of such quantum state transfer. 
\end{abstract}

\maketitle

\section{Introduction}
 
There has been recent interest in graphs whose corresponding Laplacian matrix is diagonalisable by a Hadamard matrix \cite{BFK11, JKPSZ17, Khar}. Such graphs are said to be  Hadamard diagonalisable. We generalise this notion  to include complex Hadamard matrices, originally discussed in \cite{Sylvester}. Our primary motivation for considering this interesting problem is two--fold. One basic objective is to better understand the structure of graphs under non--conventional symmetry and regularity conditions by imposing particular entry--wise structure on certain eigenbases associated with their Laplacian matrices. Second, by expanding to the complex field, we can access more potential Hadamard matrices, which in turn opens the door to analysing additional graphs with even further potentially interesting properties. In this work,  we analyse  graphs that are Hadamard diagonalisable for certain popular classes of complex Hadamard matrices. We discuss the notion of continuous--time quantum walks, where the graph represents a  quantum spin network which allows for the transfer of information from one vertex to another. We are able to characterize when  $(\alpha, \beta)$--Laplacian fractional revival and perfect state transfer occurs based on the structure of the complex Hadamard matrix involved. 

The remainder of this section is devoted to definitions, notation, basic properties, examples, and graph constructions. In Section \ref{sec:Turyn} we consider graphs that are diagonalisable by a (real) Hadamard matrix, or a class of complex Hadamard matrices called Turyn Hadamard matrices. 
Equitable partitions arise naturally, coming from the structure of these classes of complex Hadamard matrices. We are also able to make connections to the Cheeger constant. 
Section \ref{sec:Butson} focuses on graphs that are diagonalisable by a  class of complex Hadamard matrices called Butson Hadamard matrices, which include as a subclass the Turyn Hadamard matrices. We give conditions on their Laplacian eigenvalues in terms of parity and multiplicity. In Section \ref{sec:Qwalks},  we provide  necessary and sufficient conditions characterising when  $(\alpha, \beta)$--Laplacian fractional revival and perfect state transfer occurs based on the structure of the complex Hadamard matrix involved.

\subsection{Definitions and Notation}
The results in this paper hold for weighted graphs with non-negative weights unless otherwise specified.
We treat unweighted graphs as weighted graphs with edge weight $1$.  We consider only simple graphs herein.  

We quickly review some basic definitions. Given a weighted graph $G$ on $n$ vertices, its corresponding \emph{adjacency matrix} is $A(G)=[a_{i,j}]\in M_n$ where $a_{i,j}$ represents the weight of the edge between vertices $i$ and $j$. The \emph{Laplacian matrix} of a simple weighted graph $G$ is $L(G)=D(G)-A(G)$, where $D(G)$ is the diagonal matrix of 
whose $(i,i)$--th entry  is  the sum of the weights of the edges incident to vertex $i$, called the \emph{weighted degree} of vertex $i$. It is well--known that $L(G)$ is positive semi--definite with the all--ones vector as a null--vector, and that the number of connected components of $G$ equals the multiplicity of the zero eigenvalue. 

\begin{definition}
A \emph{complex Hadamard matrix} $H$ is an $n\times n$ matrix, with all its entries having moduli one, satisfying
\begin{equation*}
HH^*  = n\I.
\end{equation*}
\end{definition}
Two complex Hadamard matrices $H_1$ and $H_2$ are {\sl equivalent} if 
\begin{equation*}
H_1 = M H_2 N,
\end{equation*}
where $M$ and $N$ are invertible monomial matrices (a monomial matrix is  the product of an invertible diagonal matrix and a permutation matrix) whose non--zero entries have moduli one. 
Examples of complex Hadamard matrices include (real) Hadamard matrices whose entries are $\pm 1$ and character tables of finite abelian groups.

Originally studied in \cite{Sylvester}, the concept of a  complex Hadamard matrix has other names in the literature: it is the same as a \emph{unit Hadamard matrix} defined in \cite{Craigen91}, as well as a flat  type II matrix \cite{CG10}. 
A type II matrix is the same as an \emph{inverse orthogonal matrix} \cite{Sylvester}.
Complex Hadamard matrices can also be normalised so that the first row and first column are all--ones vectors, though in the complex setting this is known as   \emph{dephased}.  
 Analogous to \cite{BFK11}, we say that a graph is \emph{complex Hadamard diagonalisable} if its Laplacian matrix can be diagonalised by some complex Hadamard matrix.

An $n\times n$ complex Hadamard matrix whose entries are all $r$--th roots of unity is called a Butson Hadamard matrix; the set of all Butson Hadamard matrices for fixed $r$ and $n$ is denoted $H(r, n)$ \cite{Butson}. In the specific case when the complex Hadamard matrix consists only of $\pm 1$ and $\pm \iu$, the matrix is called a Turyn Hadamard matrix \cite{Turyn}. Many complex Hadamard matrices of small orders can be found in \cite{Karol}.

\subsection{Basic Properties}
Some results from \cite{BFK11}  stated for Hadamard diagonalisable unweighted graphs, readily apply with little to no adaptation to the proofs, to the situation of complex Hadamard diagonalisable weighted graphs.  

\begin{lemma}\label{lem:SS4}
A weighted graph $G$ is complex Hadamard diagonalisable if and only if there is a dephased complex Hadamard matrix that diagonalises $L(G)$.
\end{lemma}

The following result states that the regularity conclusion found in \cite[Theorem 5]{BFK11} holds for complex Hadamard diagonalisable graphs. 
We say that a weighted graph $G$ is \emph{weighted--regular} if for each vertex, the weighted degree of each vertex is the same. 

\begin{theorem}\label{thm:regularity}
If $G$ is a weighted graph and its Laplacian matrix is complex Hadamard diagonalisable,  then $G$ is weighted--regular.
\end{theorem}

\begin{proof}
Suppose $G$ is complex Hadamard diagonalisable by some complex Hadamard matrix $H$. Let $\bh_k$ be the $k$--th column of $H$.  Since $\{\bh_k\}_k$ forms an eigenbasis of $L(G)$, we may write $L(G)$ as a linear combination of $\bh_k\bh_k^*$.  Since each of $\bh_k\bh_k^*$ has constant diagonal entries (all ones), $L(G)$ also has constant diagonal entries.  Therefore, $G$ is weighted--regular.
\end{proof}

Note that \cite{BFK11} only discussed the case where the Laplacian matrix, rather than the adjacency matrix, associated to a graph is Hadamard diagonalisable. For a weighted--regular graph $G$, $A(G)$ is complex Hadamard diagonalisable by $H$ if and only if $L(G)$ is diagonalisable by $H$. The following result is an adjacency matrix version of Theorem~\ref{thm:regularity}, which shows that the property of being weighted--regular is necessary for $A(G)$ to be diagonalisable by $H$. Thus there does not exist a non--regular graph $G$ such that $A(G)$ is diagonalisable by some complex Hadamard matrix. 

\begin{theorem}
If $G$ is a weighted graph and its adjacency matrix is diagonalisable by a complex Hadamard,  then $G$ is weighted--regular.
\end{theorem}

\begin{proof}
Let $A$ be the adjacency matrix of $G$, and denote its spectral radius by $r$. Since $A$ is  diagonalisable by a complex Hadamard matrix, let $v$ be an eigenvector of $A$ corresponding to $r$, such that each entry of $v$ has modulus 1.  For any vector $u$, let $|u|$ denote the nonnegative vector whose entries consist of the moduli of the entries in $u$. Observe that by the triangle inequality, we have
\[
r\bone=|rv|=|Av|\leq A|v|=A\bone,
\]
where $\bone$ denotes the all--ones vector. In particular, for each connected component of $G$, $r$ is less than or equal to the minimum row sum of the corresponding principal submatrix of $A$. A standard result from Perron--Frobenius theory \cite{seneta} states that for an irreducible nonnegative matrix $M$, the spectral radius is always bounded below by the minimum row sum, with equality holding if and only if the row sums of $M$ are all equal. Applying this result to $A$, together with the above equation, it follows that each connected component of $G$ is weighted--regular, with  weighted--degree $r$. Thus $G$ is weighted--regular. 
\end{proof} 

\subsection{Examples and Graph Constructions}

Here, we provide details on how to construct graphs that are diagonalisable by a complex Hadamard matrix, and characterise the low--dimensional cases. 

As an example, consider the complete graph $K_n$ and the complete bipartite graph $K_{\frac{n}2, \frac{n}2}$. Suppose $H$ is a dephased complex Hadamard matrix of order $n$. For $k=1, \dots, n$, let $\mathbf{h}_k$ denote the $k$--th column of $H$. Then the space of $n\times n$ matrices diagonalisable by $H$ is 
\[
\mathcal S=\operatorname{span}\{\mathbf{h}_k\mathbf{h}_k^*\,:\, k=1, \dots, n\}.
\]
Observe that $\I$ and $\J=\mathbf{h}_1\mathbf{h}_1^*$, are in $\mathcal S$, so $K_n$ is diagonalisable by $H$. 
Contrary to Hadamard matrices,  there exists a complex Hadamard matrix of size $n$, for $n\geq 1$, with the character table of $\mathbb{Z}_n$ as an example.
Hence every complete graph is complex Hadamard diagonalisable.

Furthermore, if $H$ has a column $\mathbf{h}_j \ne \bone$ whose entries are $\pm 1$, then 
\[
\frac{n}2\I+\frac12(\mathbf{h}_j\mathbf{h}_j^*-\J)
\]
is the Laplacian matrix of $K_{\frac{n}2, \frac{n}2}$, and is diagonalisable by $H$.

A Cayley graph on a finite abelian group $\Gamma$ with connection set $C\subset\mathbb{Z}_r^d\setminus \{\bzero\}$ is a graph with vertex set $\Gamma$,
and two vertices are adjacent if and only if their difference is in $C$.
The Cayley graphs on $\mathbb{Z}_2^d$ are also known as cubelike graphs \cite{cubelike1, cubelike2}, with the hypercube being the prototypical example.

\begin{example}
Let $\Gamma$ be a finite abelian group of order $n$.   
Then the transpose of its character table, $H$, is a dephased complex Hadamard matrix  whose columns are the characters of $\Gamma$.
The space $\cS$ of matrices diagonalisable by $H$ has dimension $n$.

Take the regular representation of $\Gamma$, $\{A_g : g \in \Gamma\}$; then $\bh_j$ is an eigenvector of $A_g$ with eigenvalue $\bh_j(g)$. 
Since the set $\{A_g: g\in G\}$ spans an $n$--dimensional space, it contains all matrices diagonalisable by $H$.
Hence the (directed) Cayley graphs of $\Gamma$ are the only unweighted (directed) graphs that are diagonalisable by $H$.

Let $\mathbb{Z}_2$ be the field of two elements.  Let $C\subset \mathbb{Z}_2^d$ with $\textbf{0}\notin C$.  Since a cubelike  graph  is a Cayley graph of $\ZZ_2^d$,  we recover Corollary~1 of \cite{JKPSZ17} which states that $L(G)$ is diagonalisable by the Hadamard matrix 
\begin{equation*}
\begin{bmatrix} 1&1\\1&-1\end{bmatrix}^{\otimes d}
\end{equation*}
if and only if it is a a cubelike graph.

All unweighted regular graphs on 6 or fewer vertices are Cayley graphs and are complex Hadamard diagonalisable.
\end{example}

The notation $G^c$ denotes the \emph{complement} of the unweighted graph $G$. It is straightforward to check the following result, which, restricted to the real setting, appears in \cite[Lemma 7]{BFK11}. 
 
\begin{proposition}
Let $G$ be a complex Hadamard  diagonalisable unweighted graph, diagonalised by a complex Hadamard matrix $H$. Then $G^c$ is also  diagonalised by $H$.
\end{proposition}

 Given two graphs $G_1$ and $G_2$ on $n_1$ and $n_2$ vertices, respectively, their \emph{direct product} $G_1\times G_2$ is the graph with adjacency matrix $A(G_1)\otimes A(G_2)$.

\begin{lemma}
\label{lem:tensor}
Let $H_1$ and $H_2$ be complex Hadamard matrices. 
Suppose $G_1, \ldots, G_r$ are diagonalisable by $H_1$ and $G_{r+1}, \ldots, G_s$ are diagonalisable by $H_2$.
Then the weighted graph corresponding to the adjacency matrix
\begin{equation*}
\sum_{k = 1}^r \sum_{\ell=r+1}^s w_{k,\ell} A(G_k)\otimes A(G_\ell)
\end{equation*}
is diagonalisable by $H_1\otimes H_2$, for $w_{k,\ell} \in \RR$.
\qed
\end{lemma}

\begin{remark}
Lemma~\ref{lem:tensor} can be extended to any finite number of graphs: For any $d\in \mathbb{N}$, let $H_i$ be a complex Hadamard matrix for all  $i=1, \dots d$ and, for each $i=1, \dots, d$ and for each $j=1,\dots, r$,  suppose the graph $G_{i,j}$ is  diagonalisable by $H_i$. Then the weighted graph corresponding to the adjacency  matrix
\begin{equation*}
 \sum_{j=1}^{r} \otimes_{i=1}^d w_{i,j} A(G_{i,j})
\end{equation*}
is diagonalisable by $\displaystyle \otimes_{i=1}^d H_i$, for $w_{i,j} \in \RR$. 

For $i=1,\ldots, d$, let $G_i$ be a graph where $L(G_i)$ is diagonalisable by some complex Hadamard matrix $H_i$.
The non--complete extended $p$--sum (NEPS) of the  graphs $G_1, \dots, G_d$ with basis set $\Omega\subset \mathbb{Z}_2^d\backslash \{\mathbf 0\}$, denoted $\operatorname{NEPS}(G_1, \dots, G_d; \Omega)$ is a graph with vertex set $V(G_1)\times \cdots\times V(G_d)$ and adjacency matrix
\begin{equation*}
A_\Omega=\sum_{\beta\in \Omega}A(G_1)^{\beta_1}\otimes \cdots\otimes A(G_m)^{\beta_d}.
\end{equation*}
Note that the identity matrix is diagonalisable by any complex Hadamard matrix of the same size.   Hence 
 $\operatorname{NEPS}(G_1, \dots, G_d; \Omega)$ is diagonalisable by $H_1\otimes \ldots \otimes H_d$ when $G_i$ is diagonalisable by $H_i$, for $i=1,\ldots,d$.   
 See \cite{NEPS_PST, NEPS_PGST} for work done on perfect state transfer and pretty good state transfer on NEPS. 
\end{remark}

We recall some basic operations on graphs. Let $G_1=(V_1, E_1)$ and $G_2=(V_2,E_2)$ be graphs on disjoint vertex sets. 
The \emph{union} of the graphs $G_1$ and $G_2$  is $G_1+G_2=(V_1\cup V_2, E_1\cup E_2)$.
The \emph{join} of the graphs $G_1$ and $G_2$, denoted by $G_1\vee G_2$,  is the graph obtained from $G_1+G_2$ by adding new edges between each vertex of $G_1$ and each vertex of $G_2$.

In \cite{JKPSZ17},  Johnston et al defined the \emph{merge of two graphs $G_1$ and $G_2$ of order $n$ with respect to positive weights $w_1$ and $w_2$}, denoted by 
$G_1 \tensor[_{w_1}]{\odot}{_{w_2}} G_2$, to be the graph with adjacency matrix
\begin{equation*}
\begin{bmatrix}
w_{1}A(G_1)& {{w_{2}}}A(G_2)
\\ \noalign{\vspace{2pt}}
{{w_{2}}}A(G_2) & w_{1}A(G_1)
\end{bmatrix}.
\end{equation*}
If $G_1$ and $G_2$ are unweighted graphs on the same vertex set and $E(G_1) \cap E(G_2) = \phi$ then the merge $G_1 \tensor[_{w_1}]{\odot}{_{w_2}} G_2$ with $w_1=w_2=1$ is
called the \emph{double cover} of the graph with adjacency matrix $A(G_1)+A(G_2)$, denoted by $G_1 \ltimes G_2$ in \cite{MR3434518}.   
If $G_1$ is empty, then $G_1 \ltimes G_2$ is called \emph{the bipartite double  cover} of $G_2$.

\begin{corollary}
If $A(G_1)$ and $A(G_2)$ are adjacency matrices corresponding to weighted graphs diagonalisable by a complex Hadamard matrix $H$, then 
$G_1+G_2$, $G_1 \vee G_2$ and $G_1 \tensor[_{w_1}]{\odot}{_{w_2}} G_2$ are diagonalisable by
\begin{equation*}
\begin{bmatrix} H & H \\ H &-H\end{bmatrix}.
\end{equation*}
\end{corollary}
\begin{proof}
First observe 
\begin{equation*}
  \begin{bmatrix} H & H \\ H &-H\end{bmatrix}=\begin{bmatrix} 1&1\\1&-1\end{bmatrix} \otimes H.
\end{equation*}
The corollary follows from applying Lemma~\ref{lem:tensor} to
\begin{eqnarray*}
A(G_1+G_2) &=& \begin{bmatrix} 1 & 0 \\0&0 \end{bmatrix} \otimes A(G_1)+\begin{bmatrix} 0 & 0 \\0&1 \end{bmatrix} \otimes A(G_2), \\
A(G_1\vee G_2) &=& A(G_1+G_2) +  \begin{bmatrix}0&1\\1&0\end{bmatrix} \otimes \J, \quad \text{and}\\
A(G_1 \tensor[_{w_1}]{\odot}{_{w_2}} G_2)&=&w_1\begin{bmatrix} 1 & 0 \\0&1 \end{bmatrix} \otimes A(G_1) + w_2\begin{bmatrix}0&1\\1&0\end{bmatrix} \otimes A(G_2).
\end{eqnarray*}
\end{proof}

Given two graphs $G_1$ and $G_2$ on $n_1$ and $n_2$ vertices, respectively, their \emph{direct product} $G_1\times G_2$ is the graph with adjacency matrix $A(G_1)\otimes A(G_2)$, while their \emph{Cartesian product} $G_1\square G_2$ is a graph on
vertex set $V(G_1) \times V(G_2)$ with adjacency matrix
\begin{equation*}
A(G_1) \otimes \I_{n_2} + \I_{n_1} \otimes A(G_2).
\end{equation*}
A straightforward calculation gives the following result. 
\begin{proposition}\label{prop:SS8}
Suppose $G_1$ and $G_2$ are diagonalisable by complex Hadamard matrices $H_1$ and $H_2$, respectively.
Then $G_1\times G_2$ and $G_1\square G_2$ are diagonalisable by $H_1 \otimes H_2$. 
\end{proposition}

\section{Hadamard and Turyn Hadamard Matrices}\label{sec:Turyn}
 
Recall that a complex Hadamard matrix $H$ is called a  \emph{Turyn Hadamard matrix} if its entries are in $\{\pm 1, \pm \ii\}$.   If a Turyn Hadamard matrix of order $n$ exists then $n$ is even, see \cite{Wallis} (compare this with the real case: if a Hadamard matrix of order $n$ exists then $n=2$ or $n \equiv 0 \pmod{4}$).

As a quick example, we note that Cayley graphs of $\ZZ_4^d$ are diagonalisable by the Turyn Hadamard matrix 
\begin{equation*}
\begin{bmatrix}
1&1&1&1\\
1&\ii&-1&-\ii\\
1&-1&1&-1\\
1&-\ii&-1&\ii
\end{bmatrix}^{\otimes d}.
\end{equation*}

\begin{example}
It follows from Corollary \ref{cor:small} (see below) and Proposition \ref{prop:SS8} that $K_2\square K_6$ is diagonalisable by a Turyn Hadamard matrix.
However, by Proposition 10 of \cite{BFK11}, $K_2\square K_6$ is not diagonalisable by a Hadamard matrix.

\end{example}

The graph partitioning problem concerns partitioning the vertices of a graph while minimising the size of cut edges. This problem, often with focus on  evenly balanced cuts, is connected to the Cheeger constant, 
 expander graphs,  
flow problems, and the volume of a subset of vertices in a graph (see, e.g.\  \cite{Chung}). The problem has also been considered in the context of computer science \cite{donath73}.

\begin{definition}  (Equitable  Partition) Let  $G$ be an unweighted graph.
A partition $\pi$ of the vertex set $V(G)=V_1\cup \dots \cup V_m$ is \emph{equitable}  if, for each $i,j \in \{1,2,\ldots,m\}$, there
exists a constant $v_{i,j}$ such that every vertex in $V_i$ has exactly $v_{i,j}$ neighbours in $V_j$.
\end{definition}

The \emph{quotient graph} $G/\pi$ is a directed integer-weighted graph on $m$ vertices $V_1, \dots, V_m$, with  $v_{i,j}$ being the arc weight of the arc from $V_i$ to $V_j$ (with $v_{i,j}=0$ corresponding to no arc from $V_i$ to $V_j$).  The (adjacency) quotient matrix corresponding to $G/\pi$ is then given by the $m\times m$ matrix $A(G/\pi)=[v_{i,j}]$. 

The second conclusion in \cite[Theorem 5]{BFK11} is that all the eigenvalues of the Laplacian of a Hadamard diagonalisable unweighted graph are even integers. The theorem below generalises this to Turyn Hadamard matrices (more generally, any complex Hadamard with a column whose entries are in $\{\pm 1, \pm \ii\}$), while Theorems \ref{thm:eigvalp} and \ref{thm:eigvaltwopower} give similar results for Butson  Hadamard matrices.

Given a dephased complex Hadamard matrix $H$,
let $G$ be an unweighted graph where
\begin{equation*}
L\bh_j = \lambda_j \bh_j
\qquad \text{for $j=1,\ldots,n$.} 
\end{equation*}
Permuting the columns of $H$ if necessary, we may assume that 
 $L(G) \bh_1 = 0 \bh_1$ ($\bh_1=\bone$), and that 
\begin{equation*}
0=\lambda_1\leq \lambda_2 \leq \cdots \leq \lambda_n.
\end{equation*}

Suppose $H$ has a column $\bh_k$ , $k>1$, whose entries are in $\{\pm 1, \pm \ii\}$.  We note that, for the following argument, we do not require all four values to appear in $\bh_k$.
Define the sets
\begin{eqnarray*}
\cR_{+} = \{j :\  \bh_k(j)=1\},
&\qquad & 
\cR_{-} = \{j :\  \bh_k(j)=-1\}, \\
\cI_{+} = \{j :\  \bh_k(j)=\ii\},
&\text{and} & 
\cI_{-} = \{j :\  \bh_k(j)=-\ii\}.
\end{eqnarray*}
As $\bh_k$ is orthogonal to $\bh_1=\bone$, we have
\begin{equation*} 
|\cR_{+}| = |\cR_{-}|
\quad \text{and} \quad
|\cI_{+}| = |\cI_{-}|.
\end{equation*}

For each vertex $u$, we use $\cR_{\pm}(u)$ and $\cI_{\pm}(u)$ to denote the set of neighbours of $u$ in
$\cR_{\pm}$ and $\cI_{\pm}$, respectively.
Note that, letting $d$ denote the degree of regularity, we have 
\begin{equation}
\label{degree}
|\cR_{+}(u)|+|\cR_{-}(u)|+|\cI_{+}(u)|+|\cI_{-}(u)| = d.
\end{equation}

For $u \in \cR_{+}$, the $u$--th entry of $(L(G) \bh_k)$ is
\begin{equation*}
d-|\cR_{+}(u)|+|\cR_{-}(u)|-|\cI_{+}(u)|\ii+|\cI_{-}(u)|\ii =  \lambda_k \cdot 1.
\end{equation*}
The imaginary part of the equation gives $|\cI_{+}(u)| = |\cI_{-}(u)|$.
Together with (\ref{degree}), we get
\begin{equation*}
|\cR_{-}(u)|+|\cI_{-}(u)| = \frac{\lambda_k}{2}.
\end{equation*}
For $u \in \cR_{-}$, the $u$--th entry of $(L(G) \bh_k)$ is
\begin{equation*}
-d-|\cR_{+}(u)|+|\cR_{-}(u)|-|\cI_{+}(u)|\ii+|\cI_{-}(u)|\ii = \lambda_k (-1),
\end{equation*}
together with (\ref{degree}), yields
\begin{equation*}
|\cI_{+}(u)| = |\cI_{-}(u)| 
\quad \text{and} \quad
|\cR_{+}(u)|+|\cI_{+}(u)| = \frac{\lambda_k}{2}.
\end{equation*}
Similarly, for $u \in \cI_{+}$,  we have
\begin{equation*}
d\ii -|\cR_{+}(u)|+|\cR_{-}(u)|-|\cI_{+}(u)|\ii+|\cI_{-}(u)|\ii = \lambda_k  \cdot \ii
\end{equation*}
which leads to
\begin{equation*}
|\cR_{+}(u)| = |\cR_{-}(u)|
\quad \text{and} \quad
|\cR_{-}(u)|+|\cI_{-}(u)| = \frac{\lambda_k}{2}.
\end{equation*}
Lastly, for $u \in \cI_{-}$, we have
\begin{equation*}
-d\ii-|\cR_{+}(u)|+|\cR_{-}(u)|-|\cI_{+}(u)|\ii+|\cI_{-}(u)|\ii =  \lambda_k (- \ii )
\end{equation*}
and
\begin{equation*}
|\cR_{+}(u)| = |\cR_{-}(u)|
\quad \text{and} \quad
|\cR_{+}(u)|+|\cI_{+}(u)| = \frac{\lambda_k}{2}.
\end{equation*}

We conclude that the partition $(\cR_{+} \cup \cI_{+}, \cR_{-}\cup \cI_{-})$ of the vertex set is equitable with quotient matrix 
\begin{equation*}
\begin{bmatrix}
d-\frac{\lambda_k}{2} & \frac{\lambda_k}{2}\\
\frac{\lambda_k}{2} & d- \frac{\lambda_k}{2}
\end{bmatrix}.
\end{equation*}

We summarise the above discussion in the following theorem:

\begin{theorem} \label{thm:epqmatrixeven}
Suppose $G$ is diagonalisable by a dephased complex Hadamard matrix $H$ of order $n$.
For each $\bh_k$, $k>1$, containing entries in $\{\pm1, \pm \ii\}$, $G$ has an equitable partition into two cells, each having exactly $\frac{n}{2}$ vertices, that has quotient matrix
\begin{equation*}
\begin{bmatrix}
d-\frac{\lambda_k}{2} & \frac{\lambda_k}{2}\\
\frac{\lambda_k}{2} & d- \frac{\lambda_k}{2}
\end{bmatrix}.
\end{equation*}
Moreover, $\lambda_k$ is an even integer.
\end{theorem}

Further, we define two vectors $\bh^{\cR}_k$ and $\bh^{\cI}_k$ as follows
\begin{equation*}
\bh^{\cR}_k(j) = 
\begin{cases}
\bh_k(j) & \text{if $\bh_k(j) = \pm 1$}\\
0 & \text{otherwise}
\end{cases}
\quad \text{and} \quad
\bh^{\cI}_k(j) = 
\begin{cases}
\bh_k(j) & \text{if $\bh_k(j) = \pm \ii$}\\
0 & \text{otherwise}.
\end{cases}
\end{equation*}
Then each of $\bh^{\cR}_k$ and $\bh^{\cI}_k$  is either the zero vector or is an eigenvector of $L(G)$ corresponding to  $\lambda_k$.
If $\bh_j$ is not orthogonal to both $\bh^{\cR}_k$ and $\bh^{\cI}_k$ then $\lambda_j=\lambda_k$.

\begin{corollary}
\label{cor:TurynEven}
Suppose $G$ is diagonalisable by a normalised Hadamard matrix or a dephased Turyn Hadamard matrix of order $n$.
For each $k>1$, $G$ has an equitable partition into two cells, each having exactly $\frac{n}{2}$ vertices, that has quotient matrix
\begin{equation*}
\begin{bmatrix}
d-\frac{\lambda_k}{2} & \frac{\lambda_k}{2}\\
\frac{\lambda_k}{2} & d- \frac{\lambda_k}{2}
\end{bmatrix}.
\end{equation*}
Moreover, the eigenvalues of $L(G)$ are even integers.
\qed
\end{corollary}

\begin{remark}
If $H$ is a normalised Hadamard matrix or a dephased Turyn Hadamard matrix, then
every column of $H$, except for the first one, gives an equitable partition of $G$.

Further, if $H$ is a Hadamard matrix, then these $(n-1)$ equitable partitions are distinct,
and 
the intersection of the equitable partitions obtained from $\bh_j$ and $\bh_k$,
for $2\leq j <k\leq n$,
consists of four cells of equal size.
\end{remark}

\subsection{Small Examples}

Recall that a conference matrix $C$ of order $n$ is a symmetric or antisymmetric matrix with zeros on the diagonal, $\pm1$'s on the off--diagonals, and satisfies $C^{\top} C=(n-1)\I$.  

\begin{example}
Let $C$ be a real symmetric conference matrix of order $n$.  We obtain the matrix $H$ by dephasing the Turyn Hadamard matrix $\I+\ii C$, see \cite{Szollosi}.
Then $H$ has the form
\begin{equation*}
\begin{bmatrix}
1 & 1 & 1 & 1 &  \cdots & 1\\
1 & -1 & \pm \ii & \pm \ii & \cdots & \pm \ii \\
1 & \pm \ii &\ddots &&&\\
1 & \pm \ii & &\ddots&&\\
\vdots & \vdots &&&\ddots&\\
1 & \pm \ii & &&&\ddots\\
\end{bmatrix}
\end{equation*}
Since $\bh^{\cR}_2$ is not orthogonal to $\bh_k$, for $k >2$.  We conclude that 
\begin{equation*}
\lambda_2=\lambda_3=\cdots = \lambda_n
\end{equation*}
 and $K_n$ is the only graph diagonalisable by $H$.
\end{example}
\begin{theorem}\label{thm:oddunion}
Let $k$ be a positive integer and let $G_j$, $j=1, \dots, 2k+1$, be unweighted connected graphs of order $n$. The graph $G=G_1+\cdots +G_{2k+1}$ is not diagonalisable by a Turyn Hadamard matrix, regardless of whether or not $G_j$ is for some $j$. 
\end{theorem}
\begin{proof}
Suppose $L(G)$ is diagonalisable by a complex Hadamard matrix $H$.  The eigenspace of $L(G)$ for the eigenvalue $0$ is spanned by $(2k+1)$ columns of $H$ of the form
\begin{equation*}
\bh_j=(a_{j,1}, a_{j,2}, \ldots, a_{j,2k+1})^{\top}  \otimes \bone_n,
\quad \text{for $j=1,\ldots, 2k+1$.}
\end{equation*}
The vectors $\bh_1, \ldots, \bh_{2k+1}$ are mutually orthogonal, therefore
the matrix 
\begin{equation*}
\begin{bmatrix}
a_{1,1} & a_{1,2} & \cdots & a_{1, 2k+1} \\
a_{2,1} & a_{2,2} & \cdots & a_{2, 2k+1} \\
\vdots & & \ddots &\\
a_{2k+1, 1} & a_{2k+1, 2} & \cdots & a_{2k+1, 2k+1}  
\end{bmatrix}
\end{equation*}
is a complex Hadamard matrix of order $2k+1$.  We conclude  that $H$ is neither a Turyn Hadamard matrix nor a Hadamard matrix.
\end{proof}

\begin{corollary}\label{cor:small}
All the unweighted graphs on   $8$ or fewer vertices that are diagonalisable by a Turyn Hadamard matrix or Hadamard matrix are listed below:
\begin{itemize}
    \item Order 2: $K_2$ and $K_2^c$;
    \item Order 4: $K_4$, $C_4$, $K_2+K_2$, and $K_4^c$;
    \item Order 6: $K_6$, $K_6^c$;
    \item Order 8: $K_8$, $K_{2,2,2,2}$, 
    $(C_4+C_4)^c$, $(K_{2,2}\square K_2)^c$, $K_{4,4}$, $K_4+K_4$, $K_{2,2}\square K_2$, $C_4+C_4$, and $K_8^c$.
\end{itemize}
\end{corollary}
\begin{proof}
By Corollary~\ref{cor:TurynEven}, the graphs diagonalisable by a Turyn Hadamard matrix or a Hadamard matrix must be regular and with only even Laplacian eigenvalues.
All regular graphs on two or four vertices are Hadamard diagonalisable.
Among the  regular graphs on six vertices that are listed in \cite[Observation 3]{BFK11}, only $K_6$, $K_{2,2,2}$, $K_2+K_2+K_2$, and $K_6^c$ have all eigenvalues even.  By Theorem~\ref{thm:oddunion}, we rule out $K_2+K_2+K_2$ and its complement $K_{2,2,2}$. 
For order $8$, all regular graphs with only even eigenvalues are Hadamard diagonalisable; such graphs are listed on Page~1892 of  \cite{BFK11}.
\end{proof}

\subsection{Cheeger Constant}
The Cheeger constant of a  set $S$ of vertices in an unweighted   graph $G$ is
\begin{equation*}
h_G(S) = \frac{|E(S, V(G)-S)|}{\min(\sum_{x\in S} deg(x), \sum_{y \not \in S} deg(y))},
\end{equation*}
where $|E(S, V(G)-S)|$  is the number of edges in the edge--cut $(S, V-S)$.
The {\em Cheeger constant} of $G$ is
\begin{equation*}
h_G = \min_{S \subset V(G)} h_G(S).
\end{equation*}
The Cheeger inequality  states that \cite[Chapter~2]{Chung}
\begin{equation}\label{Cheegerineq}
\frac{\gamma_2}{2} \leq h_G \leq \sqrt{2\gamma_2},
\end{equation}
where $\gamma_2$ is the second smallest eigenvalue of the normalised Laplacian matrix, $\cL(G)=D^{-1/2}LD^{1/2}$, of $G$.

Since $G$ is $d$--regular, its normalised Laplacian matrix $\cL(G)$ equals $\frac{1}{d} L(G)$ and the second smallest eigenvalue of $\cL(G)$ is $\gamma_2=\frac{\lambda_2}{d}$.

If  $G$ is diagonalisable by a complex Hadamard matrix and the entries of $\bh_2$ are in $\{\pm 1, \pm \ii\}$, then the Cheeger constant of the set $S = \cR_{+} \cup \cI_{+}$ in $G$ is
\begin{equation*}
h_G(S) = \frac{ | S| \frac{\lambda_2}{2}}{|S| d} = \frac{\gamma_2}{2}.
\end{equation*}
Hence the lower bound on the Cheeger  constant $h_G \geq \frac{\gamma_2}{2}$ is tight, see \cite{Chung}.

\begin{proposition} \label{cheeger}
If $L(G)$ is diagonalisable by a Hadamard matrix or a Turyn Hadamard matrix, then 
\begin{equation*}
h_G = \frac{\gamma_2}{2}.
\end{equation*}
\end{proposition}
 
   In fact, both of the inequalities in (\ref{Cheegerineq}) are tight for some families of graphs: The inequality $h_G \geq \frac{\gamma_2}{2}$ is tight for all cubelike graphs. Hypercubes, which are a family of cubelike graphs, are such examples. The inequality $h_G \leq \sqrt{2\gamma_2}$ is tight for even cycles within constant factors; see  \cite[Chapter~16]{Luca} for both cases.

 Related to the Cheeger constant is a notion of {\em edge density}, which appeared in \cite{FKP}. For any subset $S$ of vertices from a graph $G$, we let
 \[ \rho(S) = \frac{|V(G)| |E(S, V(G)-S)|}{|S||V(G)-S|},\]
 and refer to $\rho(S)$ as the edge density of $S$ in the graph $G$. The minimum edge density is defined to be 
 \[  \min_{S \subset V(G)} \rho(S), \]
 and is known to be an upper bound for the algebraic connectivity (the second smallest   Laplacian eigenvalue) of $G$.
 It is not difficult to observe that for any graph $G$ satisfying the hypothesis of Proposition \ref{cheeger} also satisfies 
 \begin{equation}\label{cheq}
     \lambda_2 = \min_{S \subset V(G)} \rho(S).
 \end{equation}
 That is, the algebraic connectivity of $G$ is equal to the minimum edge density of $G$. In fact, for any such graph $G$, since $G$ is $d$--regular and when $S=\cR_{+} \cup \cI_{+}$ it follows that $2dh_G(S)=\rho(S)$. Therefore, 
 \[ \lambda_2 = 2dh_G = 2dh_G(S) = \rho(S), \]
 which verifies (\ref{cheq}). Furthermore, any such graph will also satisfy additional regularity constraints (see the  conditions labelled (A) and (B) in \cite{FKP}), which are equivalent to the resulting quotient matrix having the form as given in Theorem \ref{thm:epqmatrixeven} where $\lambda_k$ refers to the algebraic connectivity of $G$.

 \section{Butson Hadamard matrices}\label{sec:Butson}
 
Recall that a Butson Hadamard matrix $H\in H(r,n)$ is an $n\times n$ matrix satisfying $HH^*=n\I$ with all entries being $r$--th roots of unity.

\begin{lemma}
\label{lem:partition}
Let $\zeta=e^{\frac{2\pi \iu}{r}}$, for some positive integer $r$, and $\sum_{j=0}^{r-1}a_j\zeta^j = \lambda \in \QQ$.
If $r$ is a prime number, then $a_1=\cdots =a_{r-1}$.  If $r=2^m$, for some positive integer $m$, then $a_j=a_{\frac{r}{2}+j}$ for $j=1,\ldots, \frac{r}{2}-1$.
\end{lemma}
\begin{proof}
First define the polynomial 
\begin{equation*}
p(x) = a_{r-1}x^{r-1} + a_{r-2} x^{r-2}+\cdots+a_1 x + a_0-\lambda.
\end{equation*}
Suppose $r$ is a prime number.  The $r$--th cyclotomic polynomial is
\begin{equation*}
\Phi_r(x)=x^{r-1}+x^{r-2}+\cdots + x+1.
\end{equation*}
It follows from $p(\zeta)=0$  that $\Phi_r(x)$ is a factor of $p(x)$ and
\begin{equation*}
a_{r-1}=a_{r-2}=\cdots=a_1=a_0-\lambda.
\end{equation*}
Now suppose $r=2^m$, for some positive integer $m$.  
The $r$--th cyclotomic polynomial is
\begin{equation*}
\Phi_r(x)=x^{\frac{r}{2}}+1 .
\end{equation*}
As $p(\zeta)=0$, there exists polynomial $\sum_{j=0}^{\frac{r}{2}-1} b_j x^j$ satisfying
\begin{equation*}
p(x) = \Phi_r(x) \big(\sum_{j=0}^{\frac{r}{2}-1} b_j x^j\big) =  \big(\sum_{j=0}^{\frac{r}{2}-1} b_j x^j\big) +  \big(\sum_{j=0}^{\frac{r}{2}-1} b_j x^{\frac{r}{2}+j}\big).
\end{equation*}
We conclude that
$a_j=a_{\frac{r}{2}+j}=b_j$, for $j=1,\ldots,\frac{r}{2}-1$.
\end{proof}

For Theorems~\ref{thm:eigvaltwopower} and \ref{thm:eigvalp}, we recall that all rational eigenvalues of an integer--valued matrix are in fact  integers. 

The following result generalises Corollary~\ref{cor:TurynEven} for higher powers of two.

\begin{theorem}\label{thm:eigvaltwopower}
Let $G$ be an integer--weighted graph.  If $L(G)$ is diagonalisable by a Butson Hadamard matrix $H$ in $H(2^m,n)$, for some positive integer $m$, then
all integer eigenvalues of $L(G)$ are even.
\end{theorem}
\begin{proof}
We may assume $H$ is dephased.
Let $\lambda$ be an integer eigenvalue of $L(G)$ and $\bh$ be the column of $H$ satisfying $L(G) \bh = \lambda \bh$.
Let $\zeta=e^{\frac{2\pi \ii}{2^m}}$.   
For $j=0,\ldots, 2^m-1$,
let $X_j=\{s : \bh(s)=\zeta^j\}$.

For $j=0,\ldots, 2^m-1$, let
\begin{equation*}
a_j = \sum_{s\in X_j} L(G)_{1,s}.
\end{equation*}
Then the first entry of $L(G)\bh$ is 
\begin{equation*}
(L(G)\bh)_1 = \sum_{j=0}^{2^m-1} a_j \zeta^j = \lambda.
\end{equation*}
It follows from Lemma~\ref{lem:partition} that 
\begin{equation*}
\lambda = a_0+a_{2^{m-1}} \zeta^{2^{m-1}} +  \Phi_{2^m}(\zeta) \sum_{j=1}^{2^m-1} a_j \zeta^j = a_0-a_{2^{m-1}}.
\end{equation*}

Since $L(G)\bone=0 \cdot \bone$,
the first entry of $L(G)\bone$ is
\begin{equation*}
 \sum_{j=0}^{2^m-1}a_j = a_0 + a_{2^{m-1}} +2\sum_{j=1}^{2^{m-1}-1} a_j =0.
\end{equation*}
We conclude that $\lambda = a_0-a_{2^{m-1}} = -2\sum_{j=1}^{2^{m-1} }a_j$ is even.

\end{proof}

In the following theorem, we give a lower bound on the multiplicity of the integer eigenvalues other than zero for a graph diagonalisable by a Butson Hadamard matrix  $H$ in $H(p,n)$. 
Furthermore, we show that any integer eigenvalue is a multiple of $p$.

In the real setting, if a graph is Hadamard diagonalisable,  $G$ has an equitable partition corresponding to each $\{1,-1\}$--eigenvector of $L(G)$ \cite{BFK11}. In the case of Butson  Hadamard diagonalisable graphs, the proof of Theorem \ref{thm:eigvalp} leads to a similar result: the vertex set of the corresponding graph $G$ has an equitable partition based on the given eigenvector.

\begin{theorem}\label{thm:eigvalp}
Let $G$ be an integer--weighted graph on $n$ vertices.  
Suppose $L(G)$ is diagonalisable by a Butson Hadamard matrix $H$ in $H(p,n)$, for some prime number $p$. 
If $L(G)$ has  a non--zero integer eigenvalue $\lambda$, then
\begin{enumerate}[(i)]
\item
$\lambda$ is divisible by $p$,
\item
$\lambda$ has multiplicity at least $p-1$, and
\item
$G$ has an equitable partition with $p$ parts of equal size.
\end{enumerate}
\end{theorem}

\begin{proof}

We define $\bh$, $X_0, X_1, \ldots, X_{p-1}$ and $a_0, a_1, \ldots, a_{p-1}$ as in the proof of Theorem~\ref{thm:eigvaltwopower}, and let $\zeta = e^{\frac{2\pi \ii}{p}}$.
Applying Lemma~\ref{lem:partition} to
\begin{equation*}
(L(G)\bh)_1 = \sum_{j=0}^{p-1} a_j \zeta^j = \lambda
\end{equation*}
yields $a_1=\cdots=a_{p-1}$.   As $L(G)$ has zero row sums, we have $a_0=-(p-1)a_1$ and
\begin{equation*}
\lambda = a_0 -a_1 + a_1\Phi_p(\zeta) =-p a_1.
\end{equation*}

For $j=0,\ldots, p-1$, let $\mathbf{y}_j$ be the characteristic vector of the subset $X_j$ of $V(G)$ so
\begin{equation*}
\bh = \sum_{j=0}^{p-1} \zeta^j \mathbf{y}_j = \sum_{j=0}^{p-2} \zeta^j (\mathbf{y}_j-\mathbf{y}_{p-1})
\end{equation*}
where the last equality results from $\zeta^{p-1} = -\sum_{j=0}^{p-2} \zeta^j$.
Now
\begin{equation*}
\big(L(G) - \lambda \I \big) (\mathbf{y}_j-\mathbf{y}_{p-1})
\end{equation*}
is a vector with integer entries.  
So each entry of $(L(G) - \lambda \I) \bh$ is a linear combination of $1, \zeta, \zeta^2, \ldots, \zeta^{p-2}$ with integer coefficients.
Since $1, \zeta, \zeta^2, \ldots, \zeta^{p-2}$ are linearly independent over $\QQ$, we conclude that 
\begin{equation*}
\big(L(G) - \lambda \I \big) (\mathbf{y}_j-\mathbf{y}_{p-1}) = \bzero, \quad \text{for $j=0,\ldots, p-2$.}
\end{equation*}
Hence the eigenspace for the eigenvalue $\lambda$ contains at least $(p-1)$ linearly independent vectors $\{\mathbf{y}_j-\mathbf{y}_{p-1}\}_{j=0}^{p-2}$.

For vertex $u \in X_k$, let $a_{k,j}(u)$ be the sum of the weights of edges incident with $u$ and a vertex in $X_j$.
Then
\begin{equation*}
(L(G)\bh)(u) =\sum_{j=0}^{p-1} a_{k,j}(u)\zeta^j = \lambda \zeta^k.
\end{equation*}
By Lemma~\ref{lem:partition} and the fact that $(L(G)\bone)(u) = \sum_{j=0}^{p-1}a_{k,j}(u)=0$, we have
\begin{equation*}
a_{k,j}=
\begin{cases}
\frac{-\lambda}{p} & \text{if $j\neq k$}\\
\frac{(p-1)\lambda}{p} &\text{if $j=k$}
\end{cases}
\end{equation*}
which is independent of the vertex $u$.  Therefore $X_0, X_1, \ldots, X_{p-1}$ is an equitable partition of $G$.
\end{proof}
 
\begin{remark}
Let $C_r$ be the unweighted cycle of length $r$.  Then $L(C_r)$ is diagonalisable by the character table of $\ZZ_r$ which belongs to $H(r,r)$ and it has spectrum
\begin{equation*}
\Big\{2-2\cos\left(\frac{2\ell\pi}{r}\right)\Big\}_{\ell=0}^{r-1}. 
\end{equation*}   
There exist many integer--weighted graphs that have irrational Laplacian eigenvalues.
\end{remark}

\section{Continuous--time Quantum Walks}\label{sec:Qwalks}

Given a graph $G$, the (Laplacian) continuous--time quantum walk on $G$ is determined by the operator
\begin{equation*}
\e^{-\ii tL(G)}.
\end{equation*}
We use $\mathbf{e}_v$ to denote the characteristic vector of $v$.

\emph{ Fractional revival}  occurs from vertex $a$ to vertex $b$ in $G$ at time $\tau$ if \begin{equation*}
\e^{-\ii \tau L(G)} \mathbf{e}_a = \alpha  \mathbf{e}_a + \beta \mathbf{e}_b,
\end{equation*}
for some complex scalars $\alpha$ and $\beta \neq 0$ with $|\alpha|^2+ |\beta|^2=1$.  We also say $(\alpha, \beta)$--fractional revival occurs.
If $\alpha=0$, then we have (Laplacian) perfect state transfer from $a$ to $b$.  Analogous definitions go through for the continuous--time walk associated with the adjacency matrix and the  corresponding operator $e^{-\ii t  A(G)}$. 

\begin{theorem}
\label{thm:FR}
Suppose $L(G)$ is diagonalisable by a dephased complex Hadamard matrix $H$ of order $n$, say $L(G) = \frac{1}{n} H \Lambda H^*,$ where  $ \Lambda $ is the  diagonal matrix with $ \Lambda_{j,j} = \lambda_j$, for $j=1,\ldots, n$.
Then $(\alpha, \beta)$--(Laplacian) fractional revival occurs from vertex $a$ to vertex $b$ in $G$ at time $\tau$ if and only if, for $j=1,\ldots, n$,
\begin{enumerate}
\item
$H_{a,j} = \pm H_{b,j}$,   and
\item
$\e^{-\ii \tau \lambda_j} = 
\begin{cases}
1 & \text{if $H_{a,j}=H_{b,j}$},\\
\alpha - \beta & \text{if $H_{a,j}=-H_{b,j}$,}
\end{cases}
$
\end{enumerate}
In this case, there exists a real number $\gamma$ such that 
\begin{equation*}
\alpha =  \cos \gamma \ \e^{\ii \gamma},
\quad
\beta =- \ii \sin \gamma\ \e^{\ii \gamma},
\quad \text{and} \quad 
\alpha - \beta = \e^{2\gamma \ii}.
\end{equation*}
\end{theorem}
\begin{proof}
We have 
\begin{equation*}
L(G) = \frac{1}{n} H \Lambda H^* 
\quad \text{and} \quad 
\e^{-\ii \tau L(G)} = \frac{1}{n} H \e^{-\ii \tau \Lambda} H^*.
\end{equation*}
Now $\e^{-\ii \tau L(G)}  \mathbf{e}_a = \alpha \mathbf{e}_a + \beta \mathbf{e}_b$ if and only if
\begin{equation}
\label{eqn:FR}
 \e^{-\ii \tau \Lambda} H^* \mathbf{e}_a = \alpha H^* \mathbf{e}_a + \beta H^* \mathbf{e}_b,
\end{equation}
Since $\lambda_1=0$ and $H$ is dephased, the first entry of (\ref{eqn:FR}) gives
\begin{equation*}
\alpha + \beta=1.
\end{equation*}
It follows from $|\alpha|^2+|\beta|^2 = 1$ and $\alpha+\beta=1$ that $\alpha = \cos \gamma\ \e^{\ii \gamma}$ and $\beta = -\ii \sin \gamma \ \e^{\ii \gamma}$, 
for some $\gamma \in \big(-\frac{\pi}{2}, \frac{\pi}{2} \big]$.

For $j > 1$, the $j$--th entry of (\ref{eqn:FR}) is
\begin{equation*}
 \e^{-\ii \tau \lambda_j} \overline{H_{a,j}} = \alpha \overline{H_{a,j}} + \beta \overline {H_{b,j}}
 \end{equation*}
 which simplifies to
 \begin{equation}
 \label{eqn:FR2}
 \e^{-\ii \tau \lambda_j}  = \beta H_{a,j}\overline{H_{b,j}} + \alpha.
\end{equation}

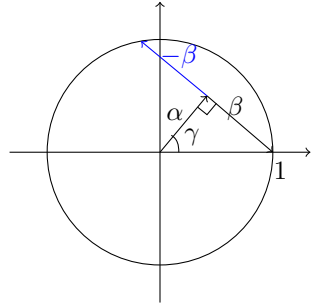
\begin{center}
\begin{tikzpicture}
\draw [->] (-2, 0) -- (2,0);
\draw [->] (0,-2) -- (0,2);
\draw (0,0) circle (1.5cm);
\draw (1.6,0) node[anchor=north]{$1$};
\draw [->] (0,0) -- (0.625,0.75);
\draw (0.2,0.3) node[anchor=south]{$\alpha$};
\draw [->]  (0.625,0.75) -- (1.5,0);
\draw (1,0.3) node[anchor=south]{$\beta$};
\draw (0.5, 0.6) -- (0.625, 0.5) -- (0.75,0.65);
\draw [->, blue]  (0.625,0.75)--(-0.25, 1.475);
\draw [blue]  (0.25,1) node[anchor=south]{$-\beta$};

\draw (0.25,0) to [out=90, in =330] (0.15,0.225);
\draw (0.2,0.2) node[anchor=west]{$\gamma$};
\end{tikzpicture}
\end{center}
Since $|H_{a,j}\overline{H_{b,j}}|=1$ and  (\ref{eqn:FR2}) implies $|\alpha + \beta \big(H_{a,j}\overline{H_{b,j}}\big)|=1$, as demonstrated in the above figure, it can be shown that
\begin{equation*}
H_{a,j}\overline{H_{b,j}} = \pm 1 
\end{equation*}
and
\begin{equation*}
\e^{-\ii \tau \lambda_j} = 
\begin{cases}
1 & \text{if $H_{a,j}=H_{b,j}$}\\
\alpha -\beta & \text{if $H_{a,j}=-H_{b,j}$}
\end{cases}
\end{equation*}
for $j=2,\ldots,n$.
\end{proof}

\begin{corollary}
\label{cor:FR}
Suppose $L(G)$ is diagonalisable by a dephased complex Hadamard matrix $H$ of order $n$.
Then $(\e^{\ii\gamma} \cos \gamma,-\e^{\ii \gamma} \ii \sin \gamma)$--(Laplacian) fractional revival occurs from vertex $a$ to vertex $b$ in $G$ at time $\tau$ if and only if there exists $\gamma \in \big(-\frac{\pi}{2}, \frac{\pi}{2} \big]$
such that, for $j=1,\ldots, n$,
\begin{enumerate}
\item
\label{cor:FR(1)}
$H_{a,j} = \pm H_{b,j}$,   and
\item
\label{cor:FR(2)}
$
-\tau\lambda_j = 
\begin{cases}
0 \pmod{2\pi} &  \text{if $H_{a,j}=H_{b,j}$},\\
2\gamma \pmod{2\pi} & \text{if $H_{a,j}=-H_{b,j}$}.
\end{cases}
$

\end{enumerate}

\end{corollary}

\begin{remark}
\label{rmk:FR}
Since a complex Hadamard diagonalisable graph $G$ is 
regular, we have $A(G) = d\I - L(G)$ and
\begin{equation*}
\e^{-\ii \tau A(G)} = \e^{-\ii d\tau} \e^{\ii \tau L(G)} = \e^{-\ii d\tau} \overline{e^{-\ii \tau L(G)}}.
\end{equation*}
Hence $G$ has $(\e^{\ii\gamma} \cos \gamma,-\e^{\ii \gamma} \ii \sin \gamma)$--(Laplacian) fractional revival from $a$ to $b$ if and only if it has \hfill\\
$(\e^{\ii (-d\tau-\gamma)} \cos \gamma,\e^{\ii (-d\tau-\gamma)} \ii \sin \gamma)$--(adjacency) fractional revival.

It follows from Proposition~5.1 of \cite{FR1}  that $a$ and $b$ are strongly cospectral in $G$: that is, if $L(G)$ has spectral decomposition $\sum_{i=1}^m \lambda_iE_i$; then  $a$ and $b$ are strongly cospectral if and only if  $E_j\mathbf{e}_a=\pm E_j\mathbf{e}_b$ for each $j=1, \dots, m$.   By Theorem~5.5 and Corollary~5.6 of \cite{FR1}  and $G$ being a regular graph,
both $A(G)$ and $L(G)$ have integral eigenvalues.
\end{remark}

The following corollary extends the proof of Theorem~4 in \cite{JKPSZ17} to complex Hadamard diagonalisable graphs.
\begin{corollary}
\label{cor:PST}
Suppose $L(G)$ is diagonalisable by a dephased complex Hadamard matrix $H$ of order $n$.
Then (Laplacian) perfect state transfer occurs from vertex $a$ to vertex $b$ in $G$ at time $\tau$ if and only if, for $j=1,\ldots, n$,
\begin{enumerate}
\item
$H_{a,j} = \pm H_{b,j}$,   and
\item
$
-\tau\lambda_j = 
\begin{cases}
0 \pmod{2\pi} &  \text{if $H_{a,j}=H_{b,j}$},\\
\pi \pmod{2\pi} & \text{if $H_{a,j}=-H_{b,j}$}.
\end{cases}
$
\end{enumerate}
\end{corollary}

We extend Theorem~2.4 of \cite{Cayley} to (Laplacian) fractional revival here.
\begin{corollary}
Let $G$ be a Cayley graph on the finite abelian group $\Gamma$ with connection set $C$.
Then $(\alpha, \beta)$--(Laplacian) fractional revival occurs in $G$ from $a$ to $b$ at time $\tau$ if and only if
the following three conditions hold: 
\begin{enumerate}
\item
The eigenvalues of $L(G)$ are integers;
\item
$a-b$ has order two;
\item
$\e^{-\ii \tau \lambda_j} =\alpha + \chi_{j}(a-b) (1-\alpha)$,
for $j \in \Gamma$.
\end{enumerate}
\end{corollary}
\begin{proof}
By Remark~\ref{rmk:FR}, $L(G)$ has integral eigenvalues if $G$ admits fractional revival.

For $j \in \Gamma$, let $\chi_j$ be the character of $\Gamma$ indexed by $j$.  We can view $\chi_j$ as a column of $H$.
Suppose the first column of $H$ corresponds to the trivial character and the first row corresponds to the identity in $\Gamma$.
Condition~(1) of Theorem~\ref{thm:FR} is equivalent to $\chi_{j}(a-b) \in \{-1,1\}$ for all $j\in \Gamma$.  That is, $a-b$ has order $2$.
Condition~(2) of Theorem~~\ref{thm:FR} is equivalent to 
\begin{equation*}
\e^{-\ii \tau \lambda_j} =\alpha + \chi_{j}(a-b) (1-\alpha),
\end{equation*}
for $j \in \Gamma$.
\end{proof}

\begin{example}
For  $n\geq 3$, the cocktail party graph $(n K_2)^c$ (i.e., the graph complement of the ladder rung graph $n K_2$) is diagonalisable by
the character table of $\ZZ_{2n}$.  It admits 
(Laplacian) fractional revival from vertex $a$ to $n+a$ at time $\frac{\pi}{n}$
with $\gamma = -\frac{\pi}{n}$; see \cite{FR2}.  
\end{example}

We apply Corollary~\ref{cor:FR} to $G_1\ltimes G_2$ where both $G_1$ and $G_2$ are diagonalisable by the same complex Hadamard matrix.  We assume $G_1$ and $G_2$ have the same vertex set $V$ and use $V\times \ZZ_2$ to denote the vertex set of $G_1 \ltimes G_2$.
\begin{corollary}
Let $G_1$ and $G_2$ be graphs  diagonalisable by 
a dephased complex Hadamard matrix $H$ of order $n$. 
Let 
\begin{equation*}
L(G_1) \bh_j = \lambda_j \bh_j
\quad \text{and} \quad
L(G_2) \bh_j = \mu_j \bh_j, \quad
\text{for $j=1,\ldots,n$,}
\end{equation*}
and let $d_2$ be the degree of $G_2$.
Then $G_1 \ltimes G_2$ has $( \e^{\ii\gamma}\cos \gamma, -\e^{\ii \gamma} \ii\sin\gamma )$--fractional revival from $(a,0)$ to $(a,1)$ at time $\tau$ if 
and only if 
\begin{equation*}
\gamma = -d_2\tau \pmod{\pi}
\end{equation*} 
and 
\begin{equation*}
\tau\lambda_j+\tau\mu_j = \tau\lambda_j-\tau\mu_j =0 \pmod{2\pi}
\end{equation*}
for $j=1,\ldots,n$.
\end{corollary}
\begin{proof}
The Laplacian matrix of $G_1\ltimes G_2$ 
\begin{equation*}
\begin{bmatrix}
L(G_1)+d_2\I & -A(G_2)\\-A(G_2) & L(G_1)+d_2\I
\end{bmatrix}
\end{equation*}
satisfies
\begin{equation*}
L(G_1\ltimes G_2) \begin{bmatrix}\bh_j\\\bh_j\end{bmatrix} = (\lambda_j+\mu_j)\begin{bmatrix}\bh_j\\\bh_j\end{bmatrix} 
\quad \text{and} \quad
L(G_1\ltimes G_2) \begin{bmatrix}\bh_j\\-\bh_j\end{bmatrix} = (\lambda_j-\mu_j+2d_2)\begin{bmatrix}\bh_j\\-\bh_j\end{bmatrix}.
\end{equation*}
As $L(G_1\ltimes G_2)$ is diagonalisable by
\begin{equation*}
\widehat{H}=\begin{bmatrix} H & H\\H & -H\end{bmatrix},
\end{equation*}
condition~(\ref{cor:FR(1)}) of Corollary~\ref{cor:FR} holds for $\widehat{H}$.
Condition~(\ref{cor:FR(2)}) of Corollary~\ref{cor:FR} holds if and only if
\begin{equation*}
\begin{cases}
&-\tau(\lambda_j+\mu_j) = 0 \pmod{2\pi},\\
&-\tau(\lambda_j-\mu_j+2d_2) = 2\gamma \pmod{2\pi}
\end{cases}
\qquad \text{for $j=1,\ldots,n$.}
\end{equation*}
Since $\lambda_1=\mu_1=0$, the second equation gives $2d_2\tau = -2\gamma \pmod{2\pi}$,
hence
\begin{equation*}
\tau\lambda_j+\tau\mu_j = \tau\lambda_j-\tau\mu_j =0 \pmod{2\pi}
\end{equation*} 
for all $j$.
\end{proof}

\begin{example}
The bipartite double cover of $K_n$ is ${K_n}^c\ltimes K_n$.  
For $n\geq 3$, (Laplacian) fractional revival occurs from $(a,0)$ to $(a,1)$ at time $\tau=\frac{2\pi}{n}$ with $\gamma=\frac{2\pi}{n}$.
\end{example}

\begin{example}
Let $G_2$ be the $(2m+1)$--cube and $G_1={G_2}^c$, for $m\geq 1$.  Then both
$G_1$ and $G_2$ are diagonalisable by
\begin{equation*}
H=\begin{pmatrix}
1&1\\1&-1
\end{pmatrix}
^{\otimes (2m+1)}.
\end{equation*}
The spectrum of $L(G_2)$ is $\{2s\ : \ s=0,\ldots,2m+1\}$, see \cite{FR2}.   
For $j=1,\ldots, 2^{2m+1}$, we have $\lambda_j+\mu_j=2^{2m+1}$ and
\begin{equation*}
\lambda_j-\mu_j \in \{2^{2m+1} - 4s\ : \ s=0,\ldots,2m+1\}.
\end{equation*}
Since $G_2$ has degree $d_2=2m+1$, the double cover of $K_n$ given by
$G_1\ltimes G_2$ has  (Laplacian) perfect state transfer ($\gamma=\frac{\pi}{2}$) from vertex $(a,0)$
to $(a,1)$ at time $\frac{\pi}{2}$.
\end{example}

\section*{Acknowledgements}
 
S.F.\ is supported by NSERC Discovery Grant number RGPIN--2019--03934. 
S.K.\ is supported by NSERC Discovery Grant number RGPIN--2019--05408. 
S.N.\ is supported by NSERC Discovery Grant number RGPIN--2019--05275. 
S.P.\ is supported by NSERC Discovery Grant number  RGPIN--2019--05276, the Canada Foundation for Innovation grant number  35711, and the Canada Research Chairs grant number 231250. S.P.\ also acknowledges support from the Brandon University Research Committee.

\bibliographystyle{plain}
\bibliography{references}
\end{document}